\documentclass[11pt]{amsart}
\usepackage{graphicx}
\usepackage{amssymb}
\usepackage{verbatim}
\usepackage{array}
\usepackage{latexsym}
\usepackage{enumerate}
\usepackage{amsmath}
\usepackage{amsfonts}
\usepackage{amsthm}
\usepackage[english]{babel}

\newcommand{\NN}{\mathbb N}
\newcommand{\ZZ}{\mathbb Z}

\def\cB{\mathcal B}
\def\cC{\mathcal C}

\def\cF{\mathcal F}

\def\cH{\mathcal H}

\def\cX{\mathcal X}
\def\cY{\mathcal Y}

\def\K{\mathbb{K}}

\def\rank{\mbox{\rm rank}}

\def\div{\mbox{\rm div}}

\def\fq{{\mathbb F}_{q}}
\def\fqq{{\mathbb F}_{q^2}}

\newcommand{\ha}{{\textstyle\frac{1}{2}}}

\newcommand{\thr}{\textstyle\frac{3}{2}}

\newcommand{\qa}{\textstyle\frac{1}{4}}

\newcommand{\eit}{\textstyle\frac{1}{8}}

\begin{document}

\title{The $a$-numbers of Fermat and Hurwitz curves}
%\thanks{Research supported by  the Italian
   % Ministry MURST, Strutture geometriche, combinatoria e loro
  %applicazioni, PRIN ???}

%\thanks{Maria Montanucci is with the Dipartimento di Matematica, Informatica ed Economia, Universit\`a degli Studi della Basilicata, Viale dell'Ateneo lucano 10, 8500 Potenza, Italy. E-mail: }
%\thanks{Pietro Speziali is with the Dipartimento di Matematica, Informatica ed Economia, Universit\`a degli Studi della Basilicata, Viale dell'Ateneo lucano 10, 8500 Potenza, Italy.  E-mail: pietro.speziali@unibas.it}

%\thanks{{\bf Keywords}: Cartier operator, Fermat curves, Hurwitz curves maximal curves}

%\thanks{{\bf Mathematics Subject Classification (2010)}: 14H55; 11T71; 11G20; 94B27
%}

\author{Maria Montanucci}

\author{Pietro Speziali}
%\date{}
\newtheorem{theorem}{Theorem}[section]
\newtheorem{proposition}[theorem]{Proposition}
\newtheorem{lemma}[theorem]{Lemma}
\newtheorem{corollary}[theorem]{Corollary}
\newtheorem{scholium}[theorem]{Scholium}
\newtheorem{observation}[theorem]{Observation}
\newtheorem{assertion}[theorem]{Assertion}
\newtheorem{result}[theorem]{Result}
{\theoremstyle{definition}
\newtheorem{definition}{Definition}
\newtheorem{example}[theorem]{Example}
\newtheorem{remark}[theorem]{Remark}
\newtheorem*{proposition*}{Proposition}
\newtheorem*{corollary*}{Corollary}
\newtheorem*{lemma*}{Lemma}

\begin{abstract}
For an algebraic curve $\cX$ defined over an algebraically closed field of characteristic $p >0$, the $a$-number $a(\cX)$ is the dimension of the space of exact holomorphic differentials on $\cX$. We compute the $a$-number for an infinite families of Fermat and Hurwitz curves. Our results apply to Hermitian curves giving a new proof for a previous result of Gross \cite{Gross}.
\end{abstract}
\maketitle
\section{Introduction}\label{Intro}
 In this paper, $\cX$ is a (projective, geometrically irreducible, algebraic) curve defined over an algebraically closed field $K$ of characteristic $p >0$. The relevant geometric properties of $\cX$ are encoded in its birational invariants, the most important being the \emph{genus}, the \emph{automorphism group}, and the $p$-\emph{rank}. The latter one is the number of independent unramified abelian $p$-extensions of the function field $K(\cX)$; equivalently, it is the dimension of the $\mathbb{F}_p$-vector space of the $p$-torsion points of the Jacobian of the curve, as well as,   the dimension of the span of the vectors in the space of holomorphic differentials that are fixed under the action of the  Cartier operator.

 In recent years, a further birational invariant of an algebraic curve related to the Cartier operator  has attracted much attention, namely the $a$-\emph{number} $a(\cX)$ defined to be the dimension of the kernel of the Cartier operator, alternatively the dimension of the space of exact holomorphic differentials. As a matter of fact, the $a$-number is also a relevant invariant of the $p$-torsion of the Jacobian of the curve; see \cite{lioort}.

  Computing the $a$-number of a curve may be a rather challenging task. In fact, the exact value of the $a$-number is known just for a few families of curves; see for instance \cite{DFREE}, \cite{EP}, \cite{farnellpries}, \cite{PriesSuzuki}, \cite{Gross}. For this purpose, a Deuring-Shafarevich type formula for the $a$-number would be useful, but it is not known whether such a generalization exists.

 In this paper, we consider Fermat and Hurwitz curves, where a \emph{Fermat curve} $\mathcal{F}_n$ is the nonsingular plane curve of affine equation $ X^n+Y^n+1 = 0$ with $p \nmid n$,
  while a \emph{Hurwitz curve} $\mathcal{H}_n$ is the nonsingular plane curve of affine equation $X^nY+Y^n+X = 0$ with $p \nmid n^2-n+1$. There is a vast literature on Fermat and Hurwitz curves, especially in the case where their number of rational points over some finite field reaches the famous Hasse-Weil bound; see \cite{aguglia-korchmaros-torres} and \cite{stich:herm}. Our main result is to compute the $a$-numbers of Fermat and Hurwitz curves for infinite values of $n$. Since  Hermitian curves are particular Fermat curves, our results apply to Hermitian curves giving a new proof for Gross' result \cite{Gross}.

   Kodama and Washio \cite{hosudatofreddo}, Gonz\'alez \cite{gonzalez}, Price and Weir \cite{CP} and previously Yui \cite{Y} obtained a few results on the ranks of the Cartier operator (and hence on the $a$-number) of Fermat curves. We will point out that our results are independent from their investigation, and in several cases are more general; see Remark \ref{finimola}.

   An essential tool in our investigation is a formula for the Cartier operator due to St\"ohr and Voloch \cite{SVCartier}. As far as we know, such a formula has not been applied previously to the study of $a$-numbers.

\section{The Cartier Operator}
Let $K(\cX)$ be the function field of a curve $\cX$ of genus $g$ defined over an algebraically closed field $K$ of characteristic $p >0$. A separating variable for $K(\cX)$ is an element $ x \in K(\cX) \setminus K(\cX)^p$. Any function $f \in K(\cX)$ can be written uniquely in the form
$$
f = u_0^p+u_1^px+\ldots +u_{p-1}^px^{p-1},
$$
where $u_i \in K(\cX)$ for $i = 0,\ldots,p-1$. Let $\Omega^1$ be the sheaf of differential $1$-forms on $\cX$. Then $\Omega^1$ is a $1$-dimensional vector space over $K(\cX)$. Hence
$$
\omega = fdx
$$
for any $\omega \in \Omega^1$. The \emph{Cartier operator}  $\cC : \Omega^1 \rightarrow \Omega^1$ is a $1/p$-linear map defined by
$$
C(fdx) = u_{p-1}dx.
$$
By a Theorem of Tate, $\cC$ does not depend on the choice of $x$; see \cite{tate}. A differential $\omega$ is holomorphic if $\div (\omega)$ is effective. The set $H^0(\mathcal{X},\Omega_1)$ of holomorphic differentials is a $g$-dimensional $K$-vector subspace of $\Omega^1$ such that $\cC(H^0(\mathcal{X},\Omega_1)) \subseteq H^0(\mathcal{X},\Omega_1)$.
 It can be shown that $\cC$ is the unique map  such that the following properties are satisfied:
\begin{enumerate}
 \item[(\rm{i})] $C (\omega_1 + \omega_2) = C (\omega_1) + C(\omega_2)$,
 \item [(\rm{ii})]$C(f^p\omega) = fC(\omega)$,
 \item[(\rm{iii})] $C(\omega) = 0$ if and only if  $\omega$ is exact,
 \item [(\rm{iv})]$C(f^{p-1}df) = df$,
 \item[(\rm{v})] $C(\omega) = \omega$ if and only if $\omega$ is logarithmic,
 \end{enumerate}
 where a differential $\omega$ is \emph{exact} if $\omega = df$ for some $f \in \K(\cX)$ (\cite{HasseWitt}), \emph{logarithmic} if $ \omega = df/f$ for $f \neq 0$. In this context, two subspaces of $H^0(\mathcal{X},\Omega_1)$ are relevant: the semisimple subspace $H^0(\mathcal{X},\Omega_1)^s $ and the nilpotent subspace $H^0(\mathcal{X},\Omega_1)^n$. The former is spanned by the holomorphic logarithmic differentials, the latter is formed by differentials $\omega$ such that there exists $n \in \NN$ for which $\cC^n(\Omega) = 0$.
 By a Theorem of Hasse and Witt, $H^0(\mathcal{X},\Omega_1) = H^0(\mathcal{X},\Omega_1)^s \oplus H^0(\mathcal{X},\Omega_1)^n$; see \cite{HasseWitt}.
 Also, by a classical result, the $p$-rank $\gamma(\cX)$ of $\cX$ equals  the dimension of $H^0(\mathcal{X},\Omega_1)^s$; see \cite{stbook}.

 The dimension $a(\cX)$ of the kernel of $\cC$ (or equivalently, the dimension of the space of exact holomorphic differentials on $\cX$) is the $a$-\emph{number} of $\cX$.
Then $0\leq a(\cX) +\gamma(\cX)\leq g$. Also, if $a(\cX) = g(\cX)$, then $\gamma(\cX) = 0$, and vice versa.
 The $a$-number arises as an invariant of the $p$-torsion subgroup of $\mathcal{J}(\cX)$;
 however, knowledge of $a(\cX)$ may give information about the Jacobian itself.
 In particular, it can be shown that the number of factors appearing in the decompostion of $\mathcal{J}(\cX)$ into simple principally polarized abelian varieties is at most $\gamma(\cX)+ a(\cX)$; see \cite{farnellpries}.

 Let $\cB = \{\omega_1,\ldots, \omega_g\}$ be a basis of $H^0(\mathcal{X},\Omega_1)$. Then for any $\omega \in H^0(\mathcal{X},\Omega_1) $,
 $$
 C(\omega) = \sum_{i = 1}^g a_{i,j}\omega_i.
 $$
 The matrix $A(\cX) = (a_{ij}^{1/p})$ is the \emph{Hasse-Witt} (or \emph{Cartier-Manin}) \emph{matrix} of $\cX$.

 The $a$-number $a(\cX)$ is the co-rank of $A(\cX)$ (or, equivalently, of $A^p(\cX) = (a_{ij})$).   Because of $1/p$-linearity, the operator $\cC^n$ is represented with respect to $\cB$ by  the matrix
  $$
  (a_{ij})(a_{ij}^{1/p})\ldots(a_{ij}^{1/p^{n-1}}).
  $$
  Also, the $p$-rank $\gamma(\cX)$ coincides with the rank of the matrix
  $$
  (a_{ij})(a_{ij}^{1/p})\ldots(a_{ij}^{1/p^{g-1}}).
  $$
The Cartier operator is related to the number $\cX(\fq)$ of points of $\cX$ over a finite field $\fq$, as the following theorem shows; see \cite{garciasaeed}.
 \begin{theorem}\label{thm:anumbermax}
 Let $\cX$ be a curve defined over a finite field with $q^2$ elements, where $q = p^n$ for some $n \in \NN$. If  $\cX$ is  $\fqq$-maximal (or $\fqq$-minimal), then $ \cC^n = 0$.
 \end{theorem}
The Cartier operator may be identically zero on $H^0(\cX,\Omega_1)$. In this case,  $\cX$ is \emph{superspecial}; see \cite{nygaard}. If $\cX$ is superspecial, then $a(\cX) = g(\cX)$. In particular, superspecial curves have zero $p$-rank.
The Jacobian $\mathcal{J}(\cX)$ of a superspecial curve is isomorphic to a product of $g$ supersingular elliptic curves.

  In \cite{SVCartier}, St\"ohr and Voloch gave a formula for the action of the Cartier operator $\cC$ on a plane curve $\cX$. We now summarize their main results.

 Let $\cY$ be a plane model of $\cX$ given by an affined equation $\cY:  F(X,Y) = 0 $ with an irreducible polynomial  $F \in K[X,Y]$ of degree $ n >3$. Then $K(\cY) = K(x,y)$ with $F(x,y) = 0$. If $\cY$ has only ordinary singularities, its \emph{canonical adjoints} are the curves of formal degree $n-3$ with at least an $(r-1)$-fold point at every $r$-fold point of $\cY$; see  \cite[Theorem 6.50]{hirschfeld-korchmaros-torres2008}. If $\cY$ is nonsingular, any curve of formal degree $n-3$ is an adjoint. In the general case, the local conditions for the canonical adjoints can be expressed in terms of the conductors of the local rings; see \cite{Gorenstein}.
 The canonical adjoints form a $g$-dimensional $K$-vector space; see \cite[Theorem 6.55]{hirschfeld-korchmaros-torres2008}. The following  theorem is due to  Gorenstein; see \cite[Theorem 12]{Gorenstein}.

  \begin{theorem}\label{thm:gorenstein}
A differential $\omega \in \Omega^1$ is holomorphic if and only if it is of the form $(h(x,y)/F_y)dx$, where $\cH : h(X,Y) = 0 $ is a canonical adjoint.
  \end{theorem}
 In particular, the Cartier operator $\cC$ acts on the canonical adjoints; see  \cite[Theorem 1.1]{SVCartier}.
\begin{theorem}\label{sv}
With the above assumptions,
\begin{equation}
\cC\Bigl(h\frac{dx}{f_y}\Bigr) = \Bigl(\frac{\partial^{2p-2}}{\partial x^{p-1}\partial y^{p-1}}(F^{p-1}h)\Bigr)^{\frac{1}{p}}\frac{dx}{F_y}
\end{equation}
 for any $h \in K(\cX)$.
\end{theorem}

\begin{remark}\label{ppower}
The differential operator $\nabla$ defined by
$$
\nabla = \frac{\partial^{2p-2}}{\partial x^{p-1}\partial y^{p-1}}
$$
has the property

$$
\nabla\Bigl( \sum_{i,j} c_{i,j}X^iY^j\Bigr) = \sum_{i,j}c_{ip+p-1,jp+p-1}X^{ip}Y^{jp}
$$
and hence $\nabla$ maps polynomials into $p$-th powers of polynomials. By Theorem \ref{sv}, the action of $\cC$ on the canonical adjoints is given by
$$
h\mapsto (\nabla(F^{p-1}h))^{\frac{1}{p}}.
$$
\end{remark}
\section{The $a$-number of Fermat curves}
Since the Fermat curve $\cF_n$ is nonsingular,
$$
\{x^iy^j \mid i+j \leq n-3\}
$$
is a basis for the space of canonical adjoints of $\cF_ n$.
Then by Theorem \ref{thm:gorenstein}, a basis for the space $H^0(\cF_n,\Omega_1)$ of holomorphic differentials on $\cF_n$ is
$$
\cB = \{(x^iy^j/ny^{n-1}) dx \mid  i+j \leq n-3\}.
$$

\begin{theorem}\label{rankfermatgen}
The rank of the Cartier operator $\cC$ on the Fermat curve $\cF_n$ equals the number of pairs $(i,j)$ with $i+j \leq n-3$ such that the system of congruences mod $p$
\begin{equation} \label{key:fermat}
 \left\{
\begin{array}{l}
n(p-1-h)+ i \equiv p-1 \\
 nk+j \equiv p-1 \\
\end{array}
\right.
\end{equation}
has a solution $(h,k)$ for $ 0\leq h \leq p-1, \: 0 \leq  k \leq h$.
\end{theorem}
\begin{proof}
By Theorem \ref{sv}, $\cC((x^iy^j/F_y)dx) = (\nabla(F^{p-1}x^iy^j))^{1/p}dx/F_y$. Therefore, we apply the differential operator $\nabla$ to
\begin{equation}\label{adjfermat}
(x^n+y^n-1)^{p-1}x^iy^j = \sum_{h=0}^{p-1}\sum_{k=0}^h\binom{p-1}{h}\binom{h}{k}(-1)^{h-k}x^{n(p-1-h)+i}y^{nk+j},
\end{equation}
for $i,j$ such that $i+j \leq n-3$.

From the formula in Remark \ref{ppower}, $\nabla((x^n+y^n-1)^{p-1}x^iy^j) \neq 0$ if and only if some $(h,k)$, with $  0 \leq h \leq p-1$ and $ 0 \leq k \leq h$, satisfies  both the following congruences mod $p$:
\begin{equation}
 \left\{
\begin{array}{l}
n(p-1-h)+ i \equiv p-1 \\
 nk+j \equiv p-1
\end{array}
\right..
\end{equation}
 Take $(i,j) \neq (i_0, j_0)$ in such a way that both $\nabla((x^n+y^n-1)x^iy^j)$ and $\nabla((x^n+y^n-1)x^{i_0}y^{j_0})$ are nonzero. We claim that they are linearly independent over $K$. To show independence we prove that for each $(h,k)$ with $ 0 \leq h \leq p-1$ and $ 0 \leq k \leq h$  there is no $(h_0,k_0)$ with $\leq h_0 \leq p-1$ and $0 \leq k_0 \leq  h_0$ such that
\begin{equation}
 \left\{
\begin{array}{l}
n(p-1-h)+ i = n(p-1-h_0) +i_0 \\
 nk+j = nk_0 + j_0
\end{array}
\right.,
\end{equation}
equivalently,
\begin{equation}
 \left\{
\begin{array}{l}
-nh+ i = -nh_0+i_0 \\
 nk+j = nk_0+ j_0.
\end{array}
\right.
\end{equation}
If $j = j_0$, then $h \neq h_0$ by $ i \neq i_0$. We may assume $h > h_0$. Then $i -i_0 = n(h-h_0) >  n$, a contradiction as $i-i_0 \leq n-3$. Similarly, $k > k_0$ yields $j-j_0 > n$, a contradiction.
\end{proof}
For the rest of this Section, $A_n:= A(\cF_n)$ is the matrix representing the $p$-power of the Cartier operator $\cC$ on the Fermat curve $\cF_n$ with respect to the the basis $\cB$.

\begin{proposition}\label{cor:ferherm}
If $n = sp+1$, $s \geq 1$, then $\rank( A_{sp+1}) = \qa s(s-1)p(p+1).$

\end{proposition}
\begin{proof}
For  $n = sp+1$, then  $i+j \leq n-3$ and system \eqref{key:fermat} modulo $p$ reads
\begin{equation} \label{key:fermat:I}
 \left\{
\begin{array}{l}
 i \equiv h \\
 j \equiv p-1-k.
\end{array}
\right.
\end{equation}

For $s = 1$, $\cF_{p+1}$ is the Hermitian curve over $\mathbb{F}_{p^2}$. From Theorem \ref{thm:anumbermax}, $\rank(A_{p+1}) = 0$. This result can also be obtained by direct computation. In fact, for $n = p+1$, we have $ i+j \leq p-2$ and since $0\leq i,j,h,k \leq p-2$, system \eqref{key:fermat:I} becomes
\begin{equation} \label{key:fermat:1}
 \left\{
\begin{array}{l}
i = h\\
 j = p-1-k. \\
\end{array}
\right.
\end{equation}
From this, $p-1 + (h-k) \leq p-2$, whence $h-k \leq 1$, a contradiction since $h \geq k$. This shows that for each $(i,j)$ there is no $(h,k)$ such that the system \eqref{key:fermat:1} is satisfied. In particular, $\rank(A_{p+1}) = 0$.

Let $n = 2p+1$. For $i+j \leq p-2$ the above argument still works. Therefore $p-1\leq i+j \leq 2p-2$, and  we need to find the solutions $(h,k)$ mod $p$ of the system
\begin{equation} \label{key:fermat:2}
 \left\{
\begin{array}{l}
 i \equiv h \\
 j \equiv p-1-k.
\end{array}
\right.
\end{equation}
Take $l,m \in \ZZ_0^+$ so that $i = lp+h$ and $j = mp+p-1-k$. Then
$$
p-1\leq (l+m+1)p+(h-k-1) \leq 2p-2.
$$
As $h-k-1 \geq -1$,
$$
 \left\{
\begin{array}{l}
 l+m+1 \geq 0\\
 l+m+1 < 2 \\
\end{array}
\right.,
$$
whence $l+m = 0$. Thus, $l = m = 0$ as $l,m \geq 0$. As we have $\sum_{i = 0}^{p-1}(i+1) = \ha p(p+1)$ choices for $(h,k)$, each yielding a different pair $(i,j)$,
$
\rank(A_{2p+1}) =  \ha p(p+1).
$

For $s \geq 3$, $\rank(A_{sp+1})$ equals $\rank(A_{(s-1)p+1})$ plus the number of $(i,j)$ such that there is $(h,k)$ solution of the system mod $p$
$$
\left\{
\begin{array}{l}
 i \equiv h\\
 j \equiv p-1-k,
\end{array}
\right.
$$
with $(s-1)p-1\leq i+j \leq sp-2$. Take $l,m $ so that $i = lp+h$ and $j = mp+p-1-k$. Then
$$
(s-1)p-1\leq (l+m+1)p+(h-k-1) \leq sp-2.
$$
Hence,
$$
 \left\{
\begin{array}{l}
 l+m+1 \geq s-1\\
 l+m+1 < s, \\
\end{array}
\right.
$$
whence, $l+m = s-2$. Since there are exactly $s-1$ different choices for $(l,m) $ and $\ha p(p+1)$ choices for $(h,k)$, we have $\ha(s-1)p(p+1)$ distinct pairs $(i,j)$.  Therefore,
$$
\rank(A_{sp+1}) = \rank(A_{(s-1)p+1}) + \ha(s-1)p(p+1).
$$
Now our claim follows by induction on $s$.
\end{proof}

\begin{remark}
 By \cite[Corollary 1]{hosudatofreddo}, a Fermat curve $\cF_n$ is superspecial if and only if $n \mid p+1$.
\end{remark}

\begin{theorem}\label{thm:anumberhermgen}
If $n = sp+1$, $s \geq 1$,  the $a$-number of the Fermat curve $\cF_{sp+1}$ is
$$
a(\cF_{sp+1}) = \qa s(s+1)p(p-1).
$$
\end{theorem}
\begin{proof}
By a direct computation, as $a(\cF_{sp+1}) = g(\cF_{sp+1})-\rank(A_{sp+1})$. Therefore, the assertion follows by Proposition \ref{cor:ferherm}.

\end{proof}
\begin{remark}
For $s= p^r$, the Fermat curve $\cF_n$ is the Hermitian curve with affine equation $X^{p^{r+1}+1}+Y^{p^{r+1}+1}+1 = 0$ and its $a$-number is equal to $ \qa p^r(p^r+1)p(p-1)$ by Theorem \ref{thm:anumberhermgen}. This agrees with Gross' result; see \cite[Proposition 14.10]{Gross}.
Since the Hermitian curve is $\mathbb{F}_{p^{2(r+1)}}$-maximal, $\cC^{r+1}=0$ by Theorem \ref{thm:anumbermax}. It should be noticed that the ranks of $\cC^n$ for $n \le r$ were determined in \cite{CP}.
\end{remark}

\begin{proposition}\label{ferm:p-1}
If $n = sp-1$, $s \geq 1$, then
$$
\rank (A_{sp-1}) = \begin{cases} \ha(p-2)(p-3), & s = 1,\\
\ha(p-2)(p-3)+p(p-2), &  s = 2,\\
3(p-1)^2, & s= 3, \\
3(p-1)^2+\qa p[(p+1)s^2+ (p-11)s-12(p-2)], & s \geq 4.
\end{cases}
$$
\end{proposition}
\begin{proof}
For $n = sp-1$, then  $i+j \leq n-3$ and system \eqref{key:fermat} modulo $p$ reads
\begin{equation} \label{ferms1meno1}
 \left\{
\begin{array}{l}
i \equiv p-(h+2) \\
 j-k \equiv p-1. \\
\end{array}
\right.
\end{equation}
For $n = p-1$,  we have  $i+j \leq p-4$ and system \eqref{ferms1meno1} mod $p$ becomes
 $$
 \left\{
\begin{array}{l}
 i \equiv p-(h+2) \\
 j \equiv k-1.  \\
\end{array}
\right.
$$
From this, $j = k-1$. Take $l \in \ZZ$ so that $ i = 2p+lp-2-h$. Then $l <0$ as $i \leq p-4$. Since
$$
i+j = k-1+(l+2)p-2-h \leq p-4
$$
can be written as
$$
(l+2)p-(h-k+3) \leq p-4,
$$
and $0\leq h-k\leq p-3$, we have $-(p+2) \leq -(h-k+3)\leq -3$. Since
$$
(l+2)p-(p+2) \leq (l+2)p-(h-k+3)\leq p-4
$$
holds whenever $l < 0$, for each $(h,k)$ there is a unique admissible pair $(i,j)$. Thus we have $\rank(A_{p-1}) = g(\cF_{p-1}) = \ha(p-2)(p-3)$.

Let $n = 2p-1$. For $i+j \leq p-4$ the above argument still works. Therefore $p-3 \leq i+j \leq 2p-4$ and  again we need to find the solutions  mod $p$ of the system
$$
\left\{
\begin{array}{l}
 i \equiv 2p-2-h \\
 j \equiv p-1 +k-1. \\
\end{array}
\right.
$$
Take $l,m$ so that $i = 2p-2+lp-h$ and $j = mp+p-1+k$. Then
$$
i = (l+2)p-(h+2) \leq 2p-4
$$
yields $l \leq 1$. Also, $l \geq -2$ as $j \geq 0$. Similarly, $-2\leq m \leq 0$. Further, as $3 \leq h-k+3 \leq p+2$, from
$$
p-3 \leq (l+m+3)p-(h-k+3) \leq 2p-4,
$$
we get
$$
1\leq (l+m+3)\leq 2.
$$
 Thus, $l+m \in \{-2,-1\}$.  If $l+m = -2$, then $(l,m) =  (-1,-1)$  and $h = k$ with $1\leq h\leq p-2$. If $l+m = -1$, then $(l,m) \in  \{(-1,0), (0,-1)\}$ and $h \geq k+1$. In the former case, $1\leq h \leq p-2$ and $0\leq k\leq h-1$, in the latter $2\leq h \leq p-1$ and $1\leq k \leq h-1$. Since any admissible $(h,k)$ yields a unique pair $(i,j)$, we have $\rank(A_{2p-1}) = \ha(p-2)(p-3)+p(p-2)$.

Let $n = 3p-1$. For $0\leq i+j \leq 2p-4$, we may still argue as before. For $2p-3 \leq i+j \leq 3p-4$ we need to consider the system mod $p$
$$
\left\{
\begin{array}{l}
 i \equiv 2p-2-h \\
 j \equiv p-1 +k-1.  \\
\end{array}
\right.
$$
This time, $-1\leq l,m\leq 1$. From
$$
2p-3 \leq (m+l+3)p-(h-k+3) \leq 3p-4,\:\: \:-3\leq-(h+k+3)\leq -(p+2)
$$
 we get $l+m \in \{-1,0\}$. If $l+m = -1$, then $(l,m) \in \{(-1,0), (0,-1)\}$ and $h = k$. If  $l+m = 0$, then $(l,m) \in \{(-1,1), (0,0), (1,-1)\}$. If $(l,m) = (-1,1)$, then $k+1\leq h \leq p-2$. If $(l,m) = (0,0)$ or $(1,-1)$, then $1\leq h \leq p-1$. Thus, $\rank(A_{3p-1}) = \rank(A_{2p-1}) + \ha(p-2)(p-2) + (p-1)(p+1)$.

 For $s \geq 4$, $\rank(A_{sp-1})$ equals $\rank(A_{(s-1)p-1})$ plus the number of distinct pairs $(i,j)$ such that there is a solution $(h,k)$  of the system mod $p$
$$
\left\{
\begin{array}{l}
 i \equiv 2p-2-h \\
 j \equiv p-1 +k-1\\
\end{array}
\right.
$$
where $(s-1)p-3 \leq i+j \leq sp-4$. Take $l,m$ such that $i = lp-2-h$ and $j = mp+ k-2$. Then  $-1\leq l,m\leq s-2$ and $m+l \in \{s-4,s-3\}$. The former condition yields $h =k$, the latter $h \geq k+1$. In fact, if $l+m = s-4$, then
\begin{equation}\label{hs4}
-(m+1)p+1 \leq h \leq (2+l)p-2.
\end{equation}
For $l > -1$ or $m\geq 0$, \eqref{hs4} holds for  any $h$, while for $l = -1$, $h \leq p-2$ and for $m = -1$,  $h \geq 1$.

If $m+l = s-3$, then
\begin{equation}\label{uff}
\left\{
\begin{array}{l}
 h \leq (2+l)p-2\\
 k \leq -(m+1)p+1. \\
\end{array}
\right.
\end{equation}
For $l >1$, \eqref{uff} holds for any $h$. Also, $h \leq p-2$ for $l = -1$. For $m \geq 0$, \eqref{uff} holds for any $k$, while $k \geq 1$ for $m =-1$. This means that $\rank(A_{sp-1}) = \rank(A_{(s-1)p-1}) + \ha(p-1)(p-2)+(s-3)p+\ha(p-1)[(s-1)p+2]$. Now, our claim follows by induction on $s$.
\end{proof}
The following result is a corollary of Proposition \ref{ferm:p-1}.
\begin{theorem}\label{chesega}
If $n = sp-1$, $s \geq 1$, the $a$-number of the Fermat curve $\cF_{sp-1}$ is
$$
a(\cF_{sp-1}) =  \qa s(s-1)p(p-1).
$$
\end{theorem}

\begin{remark}
By Theorem 3.7, the Fermat curve $\cF_{p-1}$ with $p > 3$ is ordinary. This is a special case of \cite[Theorem 6.102]{hirschfeld-korchmaros-torres2008}, stating that $\cF_n$ is ordinary if and only if $n \mid p-1$.
\end{remark}

\begin{remark}
For special values of $p$ and $n$, more $a$-numbers of Fermat curves can be obtained by combining our results with \cite[Theorems 3-4]{hosudatofreddo}. For instance, by Theorem \ref{thm:anumberhermgen}, the $a$-number of the Fermat curve $\cF_{10}$ for $p = 3$ is $18$. As a consequence of  \cite[Theorem 3]{hosudatofreddo}, $\cF_{10}$ has $a$-number $18$ also when $p = 17$. The same result holds for $p = 7$ by  \cite[Theorem 4]{hosudatofreddo}.
\end{remark}

\section{The $a$-number of Hurwitz curves}
 Since the Hurwitz curve $\cH_n$ is nonsingular,
$$
\{x^iy^j \mid i+j \leq n-2\}.
$$
is a basis for the space of canonical adjoints of $\cH_ n$. From Theorem \ref{thm:gorenstein}, a basis for the space $H^0(\cH_n,\Omega_1)$ of holomorphic differentials on $\cH_n$ is
$$
\cB' = \{(x^iy^j/(x^n+ny^{n-1})) dx \mid  i+j \leq n-2\}.
$$

\begin{theorem}
The rank of the Cartier operator $\cC$ on the Hurwitz curve $\cH_n$ equals the number of pairs $(i,j)$ with $ i+j \leq n-2$ such that the system of congruences mod $p$
\begin{equation}\label{fund:hurwitz}
 \left\{
\begin{array}{l}
nk-h+ i \equiv 0 \\
 n(h-k)+k+j \equiv p-1  \\
\end{array}
\right.
\end{equation}
has a solution $(h,k)$ for  $0\leq  h \leq p-1,\: 0\leq k \leq h$.
\end{theorem}
\begin{proof}
By Theorem \ref{sv}, $\cC((x^iy^j/F_y)dx) = (\nabla(F^{p-1}x^iy^j))^{1/p}dx/F_y$. We argue as in the proof of Theorem \ref{rankfermatgen}. This time, we apply the differential operator $\nabla$ to

\begin{equation}\label{hur:cond}
(x^ny+y^n+1)^{p-1}x^iy^j = \sum_{h=0}^{p-1}\sum_{k=0}^h\binom{p-1}{h}\binom{h}{k}x^{nk-h+p-1+i}y^{n(h-k)+k+j}.
\end{equation}
for $i+j \leq n-2$. From the formula in Remark \ref{ppower}, $\nabla((x^ny+y^n+1)^{p-1}x^iy^j )$ is non-zero if and only if some $(h,k)$ with $0 \leq h \leq p-1$, $ 0 \leq k \leq h$ satisfies both the following congruences mod $p$
$$
 \left\{
\begin{array}{l}
nk-h+ i \equiv 0 \\
 n(h-k)+k+j \equiv p-1.  \\
\end{array}
\right.
$$
 Take $(i,j) \neq (i_0,j_0)$ in such a way that $\nabla((x^ny+y^n+1)^{p-1}x^iy^j ) \neq 0$ and $\nabla((x^ny+y^n+1)^{p-1}x^{i_0}y^{j_0} ) \neq 0$. We claim that they are  independent over $K$. To this end, it is enough to show that for each $(h,k)$ there is no $(h_0,k_0)$ such that
$$
 \left\{
\begin{array}{l}
nk-h+ i = nk_0-h_0+i_0\\
 n(h-k)+k+j n(h_0-k_0)+k_0+j_0. \\
 \end{array}
\right.$$
If $h = h_0$, then $ k \neq k_0$ and $ k > k_0$ may be assumed. Then  $n(k-k_0) = i_0-i$, with $k-k_0$ a positive integer, a contradiction since  $i_0-i \leq n-2$. If  $h \neq h_0$, then $h > h_0$ may be assumed, and $h-h_0 = n(k-k_0) +(i-i_0)$. If $k = k_0$, then $n | (j-j_0) $ with $j-j_0 \leq n-2$, a contradiction. If $k \neq k_0$, then
$$
(n-1)h+k+i+j = (n-1)h_0+k_0+i_0+j_0,
$$
whence,
$$
i-i_0 = (h-h_0) + n(k_0-k) \leq n-2.
$$
Thus, $k_0-k <0$. Further, $h-h_0 \geq 2$ as $i -i_0 \leq (h-h_0)-n$.
This yields $j_0-j \geq n+1$, a contradiction.
\end{proof}
For the rest of this Section, $A'_n$ stands for the matrix representing the Cartier operator $\cC$ on the Hurwitz curve $\cH_n$ with respect to the basis $\cB'$.

\begin{proposition}\label{3marzo}
If $n = sp$, $s \geq 1$, then $\rank (A'_{sp}) =  \qa s(s-1)p(p+1)$.
\end{proposition}

\begin{proof}
If $n = sp$, then $i+j \leq sp-2$ and system \eqref{fund:hurwitz} mod $p$ reads
\begin{equation}\label{fund:hurwitz2}
 \left\{
\begin{array}{l}
 i-h \equiv 0 \\
 k+j\equiv p-1. \\
\end{array}
\right.
\end{equation}
In particular, for $n = p$, we have  $i+j \leq p-2$ and system \eqref{fund:hurwitz2} becomes
$$
 \left\{
\begin{array}{l}
 i= h\\
 j = p-1-k. \\
\end{array}
\right.
$$
From this, $h+p-1-k \leq p-2$,  whence $h \leq k-1$, a contradiction. As a consequence, there is no pair $(i,j)$ for which the above system admits a solution $(h,k)$. Thus, $\rank(A'_{p}) = 0$.

Let $n = 2p$. For  $i+j \leq p-2$, the above argument still works. Therefore,  $ p-1 \leq i+j \leq 2p-2$ and our goal is to determine for which $(i,j)$ there  is a solution $(h,k)$  of the system mod $p$
$$
 \left\{
\begin{array}{l}
 i-h \equiv 0 \\
 k+j\equiv p-1. \\
\end{array}
\right.
$$
Take $l,m \in \ZZ_0^+$ so that $i = lp+h$ and $j = mp+p-1-k$.  Then, since
$$
p-1 \leq (l+m+1)p+(h-k-1)\leq 2p-2
$$
and $0\leq h-k \leq p-1$, $1 \leq l+m+1 <2$, whence $l = m = 0$. In this way, $\ha p(p+1)$ suitable values for $(i,j)$ are obtained, whence $\rank(A'_{2p}) = \ha p(p+1)$.

Let $n = 3p$. For  $i+j \leq 2p-2$, the above argument still works. Therefore,  $ 2p-1 \leq i+j \leq 3p-2$ and we need to count the pairs $(i,j)$ for which  the system mod $p$
$$
 \left\{
\begin{array}{l}
 i-h \equiv 0 \\
 k+j\equiv p-1.
\end{array}
\right.
$$
has a solution $(h,k)$. Taking $l, m$ as before and  arguing as in the previous step, we get $l+m = 1$, that is $(l,m) \in \{(1,0), (0,1)\}$. Thus, $\rank(A'_{3p}) = \thr p(p+1)$.
For $s \geq 4$, $\rank(A'_{sp})$ equals $\rank(A'_{(s-1)p})$ plus the number of pairs $(i,j)$ with $ (s-1)p-1 \leq i+j \leq sp-2$ such that the system mod $p$
$$
 \left\{
\begin{array}{l}
 i-h \equiv 0 \\
 k+j\equiv p-1  \\
\end{array}
\right.
$$
has a solution. With our usual conventions on $l,m$, a computation shows that such pairs $(i,j)$ are obtained for $l+m = s-2$. Since we have exactly $s-1$ choices for $(l,m)$ and $\ha p(p-1)$ for $(h,k)$, each yielding a different admissible pair $(i,j)$, we have $\rank(A_{sp}) = \rank(A_{(s-1)p}) +\ha(s-1)p(p+1)$. Our  claim  follows by induction on $s$.
\end{proof}

As a corollary of Proposition \ref{3marzo}, we get the following result.

\begin{theorem}\label{primoanumberhurwitz}
If $n = sp$ for $s \geq1$, then the $a$-number of the Hurwitz curve $\cH_{sp}$ equals
$$
a(\cH_{sp}) = \qa s(s+1)p(p-1).
$$
\end{theorem}

\begin{remark}\label{fp6}
From Theorem  \ref{primoanumberhurwitz}, the Hurwitz curve $\cH_p$ is superspecial. Actually, this is also a consequence of the $\mathbb{F}_{p^6}$-isomorphism between  $\cH_p$ and $\cF_{p+1}$; see \cite[Section 12.3]{hirschfeld-korchmaros-torres2008}.
\end{remark}

\begin{remark}
By Theorems \ref{thm:anumberhermgen} and \ref{primoanumberhurwitz}, the curves $\cF_{sp+1}$ and $\cH_{sp}$ have the same $a$-number. By a straightforward computation, they also have the same genus. However, in general they are not isomorphic, as the following example shows. Let $p =3$. On the one hand, by \cite[Theorem 3.1]{aguglia-korchmaros-torres}, the Hurwitz curve $\cH_{12}$ is $\mathbb{F}_{3^{18}}$-maximal and hence its $3$-rank is equal to zero. On the other hand, a MAGMA computation shows that the Fermat curve $\cF_{13}$ has $3$-rank equal to $21$. Therefore, $\cH_{12}$ and $\cF_{13}$ are not isomorphic, although their genera and $a$-numbers coincide. \end{remark}

\begin{proposition}
If $n = sp+1$, $s \geq 1$, then
$$
\rank(A'_{sp+1}) = \qa s(s+1)p(p+1).
$$
\end{proposition}

\begin{proof}
For $n = sp+1$, then $i+j \leq n-2$ and system \eqref{fund:hurwitz} modulo $p$ reads
\begin{equation}
 \left\{
\begin{array}{l}
k-h+ i \equiv 0 \\
 h+j \equiv p-1,  \\
\end{array}
\right.
\end{equation}
for $ 0 \leq h \leq p-1$ and $k = 0\leq k \leq h$. Let $n = p+1$. Then we need to determine for which $(i,j)$ there is a solution $(h,k)$ of the system
$$
 \left\{
\begin{array}{l}
i = h-k\\
 j = p-1-h  \\
\end{array}
\right.
$$
with $i+j\leq p-1$. Since any $(h,k)$ is a solution,  then any pair $(i,j)$ is admissible, whence $\rank(A'_{p+1}) = g(\cH_{p+1}) = \ha p(p+1)$.
Let $n = 2p+1$. For $i+j \leq p-1$, the above argument can be repeated. Therefore $ p \leq i+j \leq 2p-1$. As before, we are led  to consider the system mod $p$
$$
 \left\{
\begin{array}{l}
i \equiv h-k\\
 j \equiv p-1-h.  \\
\end{array}
\right.
$$
Take $l,m$ in such a way that $i = pl+h-k$ and $j= pm-1-h$. Then any pair $(h,k)$ with $0\leq h \leq p-1$ and $0\leq k\leq h$ is a solution of the above system provided that $l+m = 1$. In this way, $p(p-1)$ distinct pairs $(i,j)$ are obtained. Thus, $\rank(A'_{2p+1}) = \ha p(p+1)+p(p+1) = \thr p(p+1)$.
 For $s \geq 3$, $\rank(A'_{sp+1})$ equals $\rank(A'_{(s-1)p+1})$ plus the number of  pairs $(i,j)$ such that the system of congruences mod $p$
$$
\left\{
\begin{array}{l}
i \equiv h-k\\
 j \equiv p-1-h  \\
\end{array}
\right.
$$
with $(s-1)p \leq i+j \leq sp-1$ has a solution $(h,k)$. There are exactly $\ha sp(p+1)$ such pairs, namely $(i,j) = (pl+h-k,pm-1-h)$ such that $l+m = s-1$, $0\leq h \leq p-1$ and $0\leq k\leq h$. Hence, $\rank(A'_{sp+1}) = \rank(A'_{(s-1)p+1})+\ha sp(p+1)$, and our result follows by induction on $s$.
\end{proof}
\begin{theorem}
If $n= sp+1$, for $ s \geq 1$, then the $a$-number of the Hurwitz curve $\cH_{sp+1}$ equals
$$
a(\cH_{sp+1}) = \qa s(s-1)p(p-1).
$$
\end{theorem}
\section{The $a$-number of Fermat curves in characteristic $2$}
Fo $p = 2$, Theorem \ref{thm:anumberhermgen} (or equivalently, Theorem \ref{chesega}) provides the $a$-number of any Fermat curve $\cF_n$; see Theorem \ref{thm:fermatchar2tot} below. Here, we give a direct proof requiring less computation.
\begin{theorem}\label{thm:fermatchar2tot}
Let $p =2$. Then $a(\cF_n) = \eit{(n^2-1)}.$
\end{theorem}
\begin{proof}
In this case, it is easier to compute $a(\cF_n)$ directly.  By Theorem \ref{sv}, we need to determine the pairs $(i,j)$ for which
\begin{equation}\label{fer:char2}
\frac{\partial}{\partial x\partial y}\Bigl((x^n+y^n+1)x^iy^j\Bigr)
\end{equation}
is equal to zero.
By a direct computation, \eqref{fer:char2} reads
\begin{equation}\label{fer:char21}
j(n+i)x^{n+i-1}y^{j-1}+ i(n+j)x^{i-1}y^{n+j-1}+ ijx^{i-1}y^{j-1}.
\end{equation}
 Since $n$ is odd, \eqref{fer:char21} is zero if and only if $i$ and $j$ are both even. Since $n-3$ is even, there are  exactly $\ha(n-1)$  even integers $r$ such that $0\leq r \leq n-3$. Also, for $i$ even, there are $\ha(n-1-i)$ choices for $j$. Therefore,
$$
a(\cF_n) = \sum_{i' = 0}^{(n-1)/2} i' = \eit{(n^2-1)}.
$$
\end{proof}
\begin{remark}\label{finimola}
For $p =2$, \cite[Corollary 2]{hosudatofreddo} states $1 \leq \rank(A_n) \leq g(\cF_n)-1$.
 Theorem \ref{thm:fermatchar2tot} gives a  better bound
$$
 3 \leq  \rank(A_n) =  g(\cF_n)-\eit{(n^2-1)} < g(\cF_n)-1.
$$
\end{remark}





\end{document}